\documentclass[12pt]{article}
\usepackage{amssymb}
\usepackage{amsfonts}
\usepackage{amsmath}
\usepackage[usenames]{color}
\usepackage{mathrsfs}
\usepackage{amsfonts}
\usepackage{amssymb,amsmath}
\usepackage{cite}
\usepackage{cases}
\usepackage{amsthm}

\usepackage[
    colorlinks=true,
    linkcolor=blue,
    citecolor=blue,
    urlcolor=blue,
    filecolor=blue,
    final,
    backref=page,
    hyperindex
]{hyperref}

\def\cl{\centerline}

\def\vs{\vspace*}

\def\Z{\mathbb{Z}}

\def\A{\mathcal{A}}

\def\C{\mathbb{C}}

\def\ni{\noindent}

\numberwithin{equation}{section}
\newtheorem{theo}{Theorem}[section]
\newtheorem{defi}[theo]{Definition}
\newtheorem{coro}[theo]{Corollary}
\newtheorem{lemm}[theo]{Lemma}
\newtheorem{exam}[theo]{Example}
\newtheorem{prop}[theo]{Proposition}
\newtheorem{clai}{Claim}
\newtheorem{case}{Case}

\newtheorem{rema}[theo]{Remark}

\newtheorem{remark}[theo]{Remark}

\begin{document}
\begin{center}
\cl{\large\bf \vs{8pt}Classification of Simple Harish-Chandra Modules over }
\cl{\large\bf \vs{8pt}the Loop Mirror Heisenberg–Virasoro Algebra}

\cl{ Haibo Chen$^*$, Xiansheng Dai, and Yucai Su}
\end{center}

\footnote {
$^*$Corresponding author: H. Chen (hypo1025@jmu.edu.cn).
}

{\small
\parskip .005 truein
\baselineskip 3pt \lineskip 3pt

\noindent{{\bf Abstract:}
The loop mirror Heisenberg-Virasoro algebra, an embedded subalgebra of the loop Heisenberg-Virasoro algebra, admits a family of interesting truncated subalgebras including those of Takiff type and  \(\mathfrak{bms}_3\) type. We give  a complete classification of  simple Harish-Chandra modules over the loop mirror Heisenberg-Virasoro algebra, whose simple modules fall into three categories: highest weight modules, lowest weight modules, and  evaluation modules of the 
intermediate series.
 As a by-product, we classify all simple Harish-Chandra modules over the truncated mirror Heisenberg–Virasoro algebras \(\mathcal{L}(n)\) for \(n\geq2\). By virtue of shift operators in the \(d\)-parameter family, we give a more streamlined proof of Theorem 3.3 from the work [Classification of simple $W_n$-modules with finite-dimensional
weight spaces,   {\it J. Reine Angew. Math.}, {\bf  720} (2016), 199-216] by Y. Billig and V. Futorny, which  states the key Billig-Futorny identity.   Furthermore, our approach can be extended to the computation of annihilators for uniformly bounded modules over some other Lie (super)algebras.
\vs{5pt}

\ni{\bf Key words:}
loop mirror Heisenberg–Virasoro algebra,   Harish-Chandra  module, weight module, shift operator.}

\ni{\it Mathematics Subject Classification (2020):} 17B10, 17B65,  17B68.}
\parskip .001 truein\baselineskip 6pt \lineskip 6pt

\tableofcontents

\section{Introduction}
In the representation theory of infinite-dimensional Lie algebras, Harish‑Chandra modules form a prominent class of weight modules, namely, weight modules with finite‑dimensional weight spaces. The classification of simple Harish-Chandra modules over the Virasoro algebra, originally conjectured by Kac (see \cite{K}), was completed in \cite{M2}. Combining results from \cite{MP} and \cite{S0}, a new approach was developed to recover this classification.

Furthermore, many generalizations of loop Virasoro algebras have been investigated in the literature. Representative examples include, but are not limited to: map (super)algebras associated with the Virasoro algebra (see \cite{CLW}), the map Virasoro algebra (see \cite{SA}), the loop Virasoro algebra (see \cite{GLZ1}), and the loop Neveu-Schwarz algebra (see \cite{WLP}), among others.

To classify all simple Harish‑Chandra modules over the Lie algebra $W_n$ of vector fields on the $n$-dimensional torus, Billig and Futorny introduced a powerful tool called {\it $\mathcal A$-cover theory} in \cite{B1,BF1,BF2}. Their results generalize Mathieu’s classification theorem for the Virasoro algebra. Since then, $\mathcal A$-cover theory has been applied to a variety of other Lie (super)algebras; see, e.g., \cite{BK,BF2,C,CLW,XL,BFIK,DCL,LG}.

The paper is organized as follows. In Section~2, we introduce some notation and definitions related to the loop mirror Heisenberg--Virasoro algebra and Harish--Chandra modules. We also recall several known classification theorems for related Lie algebras for later use.

In Section~3, we establish in Theorem~\ref{thm:3.16} a classification of uniformly bounded simple modules over the loop mirror Heisenberg--Virasoro algebra. In particular, using shift operators within the \(d\)-parameter family, we provide a simpler proof of the Billig–Futorny identity in Theorem~\ref{thm:3.9}. This approach can also be applied to compute annihilators of uniformly bounded modules over certain other Lie (super)algebras. We will continue to employ this method to address related problems in follow up work.

Finally, we present in Theorem~\ref{thm:4.6} a classification of simple Harish--Chandra modules over the loop mirror Heisenberg--Virasoro algebra.

Throughout this paper, we write $\mathbb{C}$, $\mathbb{Z}$, $\mathbb{N}$ and $\mathbb{Z}_+$ for the sets of complex numbers, integers, nonnegative integers and positive integers, respectively.
All vector spaces, modules and Lie algebras are defined over $\mathbb{C}$. In addition, all simple modules under consideration are nontrivial.
For any Lie algebra $\mathfrak{g}$, we let $U(\mathfrak{g})$ stand for its universal enveloping algebra.

\section{Preliminaries}
In this section, we define the loop mirror Heisenberg-Virasoro algebra and recall several fundamental notations that will be used throughout the paper.

\subsection{Loop mirror Heisenberg–Virasoro algebra} 

We begin by introducing the ‌loop mirror Heisenberg–Virasoro algebra‌, a natural infinite-dimensional extension constructed from the classical mirror Heisenberg–Virasoro structure.

\begin{defi} 
The loop mirror Heisenberg–Virasoro algebra  $\mathcal{L}$ is defined as the tensor product of the mirror Heisenberg–Virasoro algebra $\hat{\mathcal{L}}$ and the Laurent polynomial algebra $\mathbb{C}[t^{\pm1}]$, namely, $\mathcal{L}=\hat{\mathcal{L}}\otimes\mathbb{C}[t^{\pm1}]$. This Lie algebra admits a canonical linear basis $\big\{L_{m,i}:=L_{m}\otimes t^i,H_{r,i}:=H_{r}\otimes t^i,C_{k,i}=C_k\otimes t^i\mid m,i\in\mathbb{Z},r\in\mathbb{Z}+\frac{1}{2},k=1,2 \big\}$,  whose Lie bracket relations are explicitly specified as follows:
\begin{eqnarray*}
&&[L_{m,i},L_{n,j}]=(m-n)L_{m + n,i + j}+\delta_{m + n,0}\frac{m^{3}-m}{12}C_{1,i+j},
\\&&
[L_{m,i}, H_{r,j}] = -rH_{m+r,i+j},
 \\&& 
[H_{r,i}, H_{s,j}] = r\delta_{r+s,0}C_{2,i+j}, 
\\&& 
[C_{k,i}, \mathcal{L}] =0,
\end{eqnarray*}
where $ m, n \in \mathbb{Z}, r, s \in \frac{1}{2}+\mathbb{Z},k=1,2.$
 \end{defi}

 Clearly,
$\mathcal{L}$ contains a distinguished subalgebra isomorphic to the   mirror Heisenberg–Virasoro algebra, which corresponds precisely to the trivial-loop sector 
 $\hat{\mathcal{L}}\otimes1$. To streamline subsequent notation and avoid redundant expressions, we adopt the conventional abbreviation
$L_{m}:=L_{m}\otimes1$, $H_r:=H_r\otimes1$ and $C_k:=C_k\otimes1$ for any $m\in\mathbb{Z},r\in\mathbb{Z}+\frac{1}{2}$ and $k=1,2$.
We remark that 
the loop mirror Heisenberg--Virasoro algebra $\mathcal{L}$ can be realized as a subalgebra of the loop Heisenberg--Virasoro algebra $\mathcal{G}$ via the homomorphism $\phi\colon \mathcal{L}\to\mathcal{G}$ defined by
\[
L_{m,i}\mapsto \frac12 L_{2m,i}+\frac{1}{16}\delta_{m,0}C_{1,i},\ 
H_{r,i}\mapsto H_{2r,i},\ 
C_{1,i}\mapsto 2C_{1,i},\ 
C_{2,i}\mapsto 2C_{2,i},
\]
where $m,i\in\mathbb{Z}$, $r\in\mathbb{Z}+\frac12$.
Let $\psi\colon \mathcal{L}\to\mathcal{L}$ be a linear map defined by
\[
\psi(L_{m,i})=-L_{-m,i},\ 
\psi(H_{r,i})=-H_{-r,i},\ 
\psi(C_{1,i})=-C_{1,i},\ 
\psi(C_{2,i})=-C_{2,i},
\]
where $m,i\in\mathbb{Z}$, $r\in\mathbb{Z}+\frac12$. Then $\psi$ is an automorphism of $\mathcal{L}$.

Note that $\mathcal{L}$ is the semidirect product of the loop Virasoro subalgebra
$\mathcal{V}=\operatorname{span}\{L_{m,i},C_{1,i}\mid m,i\in\mathbb{Z}\}$
and the loop  Heisenberg subalgebra
$\mathcal{I}=\operatorname{span}\{H_{r,i},C_{2,i}\mid r\in\mathbb{Z}+\tfrac12,\,i\in\mathbb{Z}\}$.
Throughout this paper, we regard $\mathcal{V}$ and $\mathcal{I}$ as subalgebras of $\mathcal{L}$.
It is clear that $\hat{\mathcal{V}}=\operatorname{span}\{L_{m}\otimes 1,\,C_{1}\otimes 1 \mid m\in\mathbb{Z}\}$ is isomorphic to the classical Virasoro algebra.
The Lie algebra $\mathcal{L}$ carries a natural $\tfrac12\mathbb{Z}$‑grading
\[
\mathcal{L}=\bigoplus_{p\in\frac12\mathbb{Z}}\mathcal{L}_p,
\]
where
\[
\mathcal{L}_p=
\begin{cases}
\bigl(L_p\otimes\mathbb{C}[t^{\pm1}]\bigr)\oplus\delta_{p,0}\bigl(C_1\otimes\mathbb{C}[t^{\pm1}]\bigr)
\oplus\delta_{p,0}\bigl(C_2\otimes\mathbb{C}[t^{\pm1}]\bigr), & \text{if }p\in\mathbb{Z};\\[4pt]
H_p\otimes\mathbb{C}[t^{\pm1}], & \text{if }p\in\mathbb{Z}+\tfrac12.
\end{cases}
\]
This grading is given by the eigenvalues of $\operatorname{ad}(L_0)$.
Consequently, $\mathcal{L}$ admits a triangular decomposition
$\mathcal{L}=\mathcal{L}_+\oplus\mathcal{L}_0\oplus\mathcal{L}_-$, where
\[
\mathcal{L}_\pm=\bigoplus_{m\in\mathbb{Z}_+}\bigl(L_{\pm m}\otimes\mathbb{C}[t^{\pm1}]\bigr)
\oplus\bigoplus_{r\in\mathbb{N}+\frac12}\bigl(H_{\pm r}\otimes\mathbb{C}[t^{\pm1}]\bigr),
\]
\[
\mathcal{L}_0=\bigl(L_0\otimes\mathbb{C}[t^{\pm1}]\bigr)\oplus\bigl(C_1\otimes\mathbb{C}[t^{\pm1}]\bigr)
\oplus\bigl(C_2\otimes\mathbb{C}[t^{\pm1}]\bigr).
\]
Clearly, $\hat{\mathcal{L}}$ has a similar triangular decomposition.

For any $n\in\mathbb{Z}_+$, we define the truncated mirror Heisenberg-Virasoro algebra
\[
\mathcal{L}(n):=\hat{\mathcal{L}}\otimes\bigl(\mathbb{C}[t]/t^n\mathbb{C}[t]\bigr).
\]

The following proof is attributed to Professor Xiangqian Guo.

\begin{prop}\label{lemm:2.2} For $n\in\mathbb{Z}_+$,
the truncated mirror Heisenberg-Virasoro algebra satisfies
\[
\mathcal{L}(n)\cong \hat{\mathcal{L}}\otimes \bigl(\mathbb{C}[t^{\pm1}]/t^n\mathbb{C}[t^{\pm1}]\bigr).
\]
\end{prop}

\begin{proof}
One readily verifies that
\[
\frac{\mathbb{C}[t]}{t^n\mathbb{C}[t]} \cong \frac{\mathbb{C}[t+1]}{(t+1)^n\mathbb{C}[t+1]}
= \frac{\mathbb{C}[t]}{(t+1)^n\mathbb{C}[t]}
= \frac{\mathbb{C}[t]}{(t+1)^n\mathbb{C}[t^{\pm1}]\cap \mathbb{C}[t]}.
\]
Applying the second isomorphism theorem for rings yields
\[
\frac{\mathbb{C}[t]}{(t+1)^n\mathbb{C}[t^{\pm1}]\cap \mathbb{C}[t]}
\cong \frac{\mathbb{C}[t^{\pm1}]}{(t+1)^n\mathbb{C}[t^{\pm1}]}\stackrel{\text{s=t+1}}{=}\frac{\mathbb{C}[s^{\pm1}]}{s^n\mathbb{C}[s^{\pm1}]}.
\]
The proposition holds.
\end{proof}

The truncated mirror Heisenberg–Virasoro algebras $\mathcal{L}(2)$ and $\mathcal{L}(3)$ are respectively called the Takiff mirror Heisenberg–Virasoro algebra (also called mirror BMS-Kac-Moody algebra) and the $\mathfrak{bms}_3$ mirror Heisenberg–Virasoro algebra. $\mathcal{L}_{0}$ is an infinite dimensional abelian subalgebra of $\mathcal{L}$, and $\mathcal{C}={\rm span}\big\{C_k \otimes t^i\mid k=1,2,i\in\mathbb{Z}\big\}$ is the center of $\mathcal{L}$. We denote $\bar{\mathcal{L}}=\mathcal{L}/\mathcal{C}$. The loop mirror Heisenberg–Virasoro algebra $\mathcal{L}$ is not a pre-exp-polynomial algebra as defined in \cite[Definition 1.5]{BGLZ}, since Condition (P1) is not satisfied. However, it still possesses properties similar to those of pre-exp-polynomial algebras.

\subsection{Weight  module}

Let $\psi: \mathcal{L}_{0} \to \mathbb{C}$ be a linear function. We recall that a  highest weight vector in a $U(\mathcal{L})$-module $V$ is a vector $v \in V$ satisfying
$$xv=\psi(x)v,\quad \mathcal{L}_{+}v=0$$
for all $x \in \mathcal{L}_{0}$. A   $U(\mathcal{L})$-module  is referred to as a {\it highest weight module} if it is cyclically generated by a highest weight vector.

Now let $\psi\in\operatorname{Hom}(\mathcal{L}_0,\mathbb{C})$ be a one‑dimensional representation of the subalgebra $\mathcal{L}_0$.
We extend the underlying vector space $\mathbb{C}_\psi$ to a module over $\mathcal{L}_0\oplus\mathcal{L}_+$ by letting $\mathcal{L}_+$ act trivially on it.
The associated {\it Verma module} corresponding to $\psi$ is defined as
\[
F(\psi):=U(\mathcal{L})\otimes_{U(\mathcal{L}_0\oplus\mathcal{L}_+)}\mathbb{C}_\psi.
\]
It is a highest weight module with highest weight $\psi(L_0)$, and we have the weight space decomposition
$F(\psi)=\bigoplus_{p\in\frac12\mathbb{N}} F(\psi)_{\psi(L_0)-p}$.
Define
\[
\tilde{v}_\psi:=1\otimes 1_\psi,
\]
where $1_\psi$ stands for the unit element of $\mathbb{C}_\psi$.
Hence $\tilde{v}_\psi$ is a highest  weight vector of $F(\psi)$.
Observe that $F(\psi)\cong U(\mathcal{L}_-)$ as $U(\mathcal{L}_-)$‑modules.

For $\psi\in\operatorname{Hom}(\mathcal{L}_0,\mathbb{C})$, let $J(\psi)$ denote the unique maximal proper submodule of $F(\psi)$.
Then
\[
V(\psi):=F(\psi)/J(\psi)
\]
is the simple highest weight module associated with $\psi$.
It is a highest weight module of highest weight $\psi(L_0)$, with weight decomposition
$V(\psi)=\bigoplus_{p\in\frac12\mathbb{N}} V(\psi)_{\psi(L_0)-p}$.
We write $v_\psi$ for the image of $\tilde{v}_\psi$ in $V(\psi)$.

 For the mirror Heisenberg–Virasoro algebra $\hat{\mathcal{L}}$, one may similarly define the Verma module over $\hat{\mathcal{L}}$ by
\[
F(c_1,c_2,h):=U(\hat{\mathcal{L}})\otimes_{U(\hat{\mathcal{L}}_+\oplus\hat{\mathcal{L}}_0)}\mathbb{C}v_0,
\]
where $\psi(C_k)=c_k$ and $\psi(L_0)=h$, with $\hat{\mathcal{L}}_+=\hat{\mathcal{L}}\cap\mathcal{L}_+$ and $\hat{\mathcal{L}}_0=\hat{\mathcal{L}}\cap\mathcal{L}_0$.
In the same fashion, we can  define the simple highest weight module $V(c_1,c_2,h)$.

 Consider a nontrivial module $F$ over $\mathcal{L}$. The action of each central element $C_{k,i}$ on $F$ is a scalar $c_{k,i}$ for $k=1,2$, $i\in\mathbb{Z}$.
The module $F$ is said to be \emph{trivial} if the whole algebra acts trivially on $F$.
Denote
\[
F_\lambda=\bigl\{v\in F \mid L_0 v=\lambda v\bigr\},
\]
which is called the \emph{weight space} of weight $\lambda\in\mathbb{C}$. We say that $F$ is a \emph{weight module} if $F=\bigoplus_{\lambda\in\mathbb{C}}F_{\lambda}$. Set
\[
\mathrm{Supp}(F)=\bigl\{\lambda \mid F_{\lambda}\neq 0\bigr\},
\]
which is called the \emph{support} (or \emph{weight set}) of $F$.
An indecomposable weight module $F$ whose all weight spaces are one‑dimensional is called an \emph{intermediate series module}.

\begin{defi}
Let $F$ be a weight module over $\mathcal{L}$.
\begin{itemize}
\item[{\rm (1)}]
If $\dim(F_\lambda)<+\infty$ for all $\lambda\in\mathrm{Supp}(F)$, then $F$ is called a Harish‑Chandra module.

\item[{\rm (2)}]
If there exists some $K\in\mathbb{Z}_+$ such that $\dim(F_\lambda)<K$ for all $\lambda\in\mathrm{Supp}(F)$, then $F$ is called  a uniformly bounded module {\rm (}or called a cuspidal module{\rm )}.
\end{itemize}
\end{defi}

\subsection{Some known results and evaluation modules}
This subsection recalls several known results that will be applied in the corresponding study of the loop mirror Heisenberg–Virasoro algebra $\mathcal{L}$.

As observed in \cite{LPXZ}, $F_{a,b,c}$ is an intermediate series module over $\hat{\mathcal{L}}$ for some $a,b,c\in\mathbb{C}$, defined by
\[
F_{a,b,c}:=\sum_{i\in\frac12\mathbb{Z}}\mathbb{C}v_i
\]
with
\begin{equation}\label{eq:2.1}
L_m v_k=(a+mb-k)v_{m+k},\ H_r v_n=v_{n+r},\ H_r v_s=c v_{r+s},\ C_j v_k=0,
\end{equation}
where $m,n\in\mathbb{Z}, r,s\in\mathbb{Z}+\frac12, k\in\frac12\mathbb{Z}, j=1,2$. If $a,b\in\mathbb{C}$ and $c\in\mathbb{C}^*$, then the $\hat{\mathcal{L}}$-module $F_{a,b,c}$ is simple.  
Note that $\sum_{m\in\mathbb{Z}}\mathbb{C}v_m$ is an intermediate series $\hat{\mathcal{V}}$-module, denoted by $F_{a,b}$. By setting $\mathcal{I}F_{a,b}=0$, we obtain an $\hat{\mathcal{L}}$-module structure on $F_{a,b}$. The $\hat{\mathcal{L}}$-module $F_{a,b}$ is simple if and only if $a\notin\mathbb{Z}$ or $b\neq 0,-1$. Whenever $F_{a,b}$ is not simple as an $\hat{\mathcal{L}}$-module, it has a unique nontrivial subquotient. We denote by $\bar{F}_{a,b}$ the corresponding simple $\hat{\mathcal{L}}$-module or simple quotient module, following the notation used for $\hat{\mathcal{V}}$.

\begin{theo}\label{thm:2.4}{\rm (see \cite{LPXZ})}
Let $V$ be a uniformly bounded  simple module over the mirror Heisenberg–Virasoro algebra $\hat{\mathcal{L}}$. Then $V\cong F_{a,b,c}$ or $V\cong \bar{F}_{a,b}$ for some $a,b\in\mathbb{C}$ and $c\in\mathbb{C}^*$.
\end{theo}

For any $\hat{\mathcal{L}}$-module $V$ and $\lambda\in\mathbb{C}^*$, we define the evaluation $\mathcal{L}$-module $V(\lambda)$ as follows: $V(\lambda)=V$ as vector spaces, with the action of $\mathcal{L}$ on $V(\lambda)$ given by
\begin{equation*}
(x \otimes t^i)\cdot v = \lambda^{i} x v,
\end{equation*}
where $x\in\hat{\mathcal{L}}$.
Examples include the evaluation modules $F_{a,b,c}(\lambda)$, and $\bar{F}_{a,b}(\lambda)$.
Since the central elements act trivially, it is sometimes also common to denote the evaluation modules of the intermediate series of $\bar{\mathcal{L}}$.

\begin{theo}\label{thm:2.3}{\rm (see \cite{GLZ1})}
Let $V$ be a simple Harish‑Chandra module over the loop Virasoro algebra $\mathcal{V}$. Then $V$ is either a highest weight module, a lowest weight module, or an evaluation module of the intermediate series $\bar{F}_{a,b}(\lambda)$.
\end{theo}

\section{Uniformly bounded modules}
In this section, we classify the  uniformly bounded  simple modules over the loop mirror Heisenberg-Virasoro algebra.

\subsection{\texorpdfstring{$\mathcal{A}\bar{\mathcal{L}}$-}-modules}
Let $\mathcal{A}_1$ and $\mathcal{A}_2$ denote the Laurent polynomial algebras $\mathbb{C}[z^{\pm\frac{1}{2}}]$ and $\mathbb{C}[t^{\pm1}]$, respectively. Both are unital associative algebras with multiplication rules $z^rz^s=z^{r+s}$ for all $r,s\in \frac{1}{2}\mathbb{Z}$ and $t^mt^n=t^{m+n}$ for all $m,n\in\mathbb{Z}$. We further define the tensor product algebra $\mathcal{A}=\mathcal{A}_1\otimes\mathcal{A}_2$, namely, $\mathcal{A}=\mathbb{C}[z^{\pm\frac{1}{2}}]\otimes\mathbb{C}[t^{\pm1}]$.
Clearly, $\mathcal{A}$ is a unital associative algebra with multiplication law
$$(z^{p_1}\otimes t^{n_1}) (z^{p_2}\otimes t^{n_2})=z^{p_1+p_2}\otimes t^{n_1+n_2}$$
for all $p_1,p_2\in\frac{1}{2}\mathbb{Z}$ and $n_1,n_2\in \mathbb{Z}$. It is well known that the centerless mirror Heisenberg-Virasoro algebra can be realized as the semidirect product ${\rm Der}(\mathbb{C}[z^{\pm\frac{1}{2}}])\ltimes \mathbb{C}[z^{\pm\frac{1}{2}}]$.
Accordingly, we give an explicit realization of the centerless loop mirror Heisenberg-Virasoro algebra $\bar{\mathcal{L}}$ in terms of the following basis elements:
\begin{eqnarray*} 
L_{m,i}=z^{m+1}\frac{d}{dz}\otimes t^i, \quad  H_{r,i}=z^{r}\otimes  t^i,
\end{eqnarray*}
where $m,i\in\mathbb{Z}$ and $r\in\mathbb{Z}+\frac{1}{2}$.

We  now describe the structure of uniformly bounded $\bar{\mathcal{L}}$-modules that admit a compatible action of the commutative  unital algebra $\A$.

\begin{defi} {\rm (see \cite{BF1})}
A module $V$ is called an $\A\bar{\mathcal{L}}$-module if it is simultaneously a module over $\bar{\mathcal{L}}$ and the commutative unital algebra $\mathcal{A}=\mathbb{C}[z^{\pm\frac{1}{2}}]\otimes\mathbb{C}[t^{\pm1}]$, with the two structures satisfying the compatibility condition
\begin{eqnarray}\label{511}
 y(fv)=(yf)v + f(yv)  \quad \mathrm{for}\ f\in \A, y \in \bar{\mathcal{L}}, v\in V.
\end{eqnarray}
\end{defi}

Let $V$ be a weight module over $\A\bar{\mathcal{L}}$. By \eqref{511}, the action of $\A$ is compatible with the weight grading of $V$:
$$\A_\alpha V_\lambda\subset V_{\alpha+\lambda} \quad \mathrm{for}\ \alpha, \lambda\in\frac{1}{2}\Z.$$
Suppose that the $\A\bar{\mathcal{L}}$-module $V$ admits a weight space decomposition and that one of its weight spaces is finite dimensional. Since all nonzero homogeneous elements of $\A$ are invertible, all weight spaces of $V$ have the same dimension. Hence, $V$ is a free $\A$-module of finite rank. It follows that the $\A\bar{\mathcal{L}}$-module $V$ is uniformly bounded (as an $\bar{\mathcal{L}}$-module).

Assume that $V$ is a uniformly bounded $\A\bar{\mathcal{L}}$-module. Let $W=V_g$ for $g\in\frac{1}{2}\Z$, with $\mathrm{dim}(W)<\infty$.
Since $V$ is a free $\A$-module, we may write
$$V\cong\A\otimes W.$$
\begin{theo}  
\label{thm:3.2}
Assume that $V$ is a uniformly bounded $\mathcal{A}\bar{\mathcal{L}}$-module and that $V= \A\otimes W$, where $W=V_g$ for $g\in\frac{1}{2}\Z$. Then $V$ is isomorphic to an evaluation module over $\bar{\mathcal{L}}$; namely, for any $v\in V$,  $L_{m,i}v=\lambda^iL_m v,H_{r,i}v=\lambda^iH_r v$, where $m,i\in\mathbb{Z}$ and $\lambda\in\mathbb{C}^*$.
\end{theo}
\begin{proof}
For any $m,i,\alpha_2\in\Z,\alpha_1\in\frac{1}{2}\mathbb{Z},r\in\mathbb{Z}+\frac{1}{2},w\in W$, it follows from \eqref{511} that
\begin{align}
\label{3.3}
L_{m,i}\big((z^{\alpha_1}\otimes t^{\alpha_2}) w\big)=\big(L_{m,i}(z^{\alpha_1}\otimes t^{\alpha_2})\big)w+z^{\alpha_1}\otimes t^{\alpha_2}(L_{m,i} w).
 \end{align}
Let $\bar{\mathcal{L}}_{\mathcal{A}}=\bar{\mathcal{L}}\oplus(\mathcal{A}_1\otimes\mathcal{A}_2)$.
Clearly, $[1\otimes\mathcal{A}_2,U(\bar{\mathcal{L}}_{\mathcal{A}})]=0$.
Thus, the elements $1\otimes t^i$ are central in ${\bar{\mathcal{L}}}_{\A}$.
By Schur's Lemma, we may assume that $1\otimes t^i$ acts on $W$ as the scalar $\lambda_i\in\mathbb{C}$. Let $w\in W$.
For any $i\in\mathbb{Z}$ and $m,n\in\mathbb{Z}$ with $m-n\neq0$, it follows from $[L_{m,i},L_{n,j}]=(m-n)L_{m+n,i+j}$ that
\begin{align*}
\lambda_{i+j}L_{m+n} w=&\frac{1}{m-n}(L_{m,i}L_{n,j}-L_{n,j}L_{m,i})w
 \\=&\frac{1}{m-n}\Big((L_{m}\otimes t^i)(L_n\otimes t^j)-(L_n\otimes t^j)(L_{m}\otimes t^i)\Big)w
\\=&\frac{\lambda_i
\lambda_j}{m-n}\big(L_{m}L_n-L_nL_{m}\big)w
\\=&\lambda_i
\lambda_jL_{m+n} w,
\end{align*}
which implies $\lambda_i=\lambda^i$.
Then, for any $v\in V$, $i\in\mathbb{Z}$, and $r\in\mathbb{Z}+\frac{1}{2}$, \eqref{3.3} and
\begin{align*}
H_{r,i}\big((z^{\alpha_1}\otimes t^{\alpha_2}) w\big)=\big(H_{r,i}(z^{\alpha_1}\otimes t^{\alpha_2})\big)w+z^{\alpha_1}\otimes t^{\alpha_2}(H_{r,i} w).
 \end{align*}
respectively yield
$
L_{m,i}v=\lambda^iL_{m}v
$ and
$
H_{r,i}v=\lambda^iH_{r}v.
$
This completes the proof.
\end{proof}
\begin{rema}
By Theorem \ref{thm:3.2}, the classification of uniformly bounded simple $\mathcal{A}\bar{\mathcal{L}}$-modules reduces to the classification of uniformly bounded simple modules over the centerless mirror Heisenberg-Virasoro algebra.
\end{rema}

Then, combining Theorems \ref{thm:2.4} and \ref{thm:3.2}, we obtain the main result of this subsection.
\begin{theo}\label{thm:3.4}
Let $V$ be a uniformly bounded simple module over $\mathcal{A}\bar{\mathcal{L}}$. Then, as an $\bar{\mathcal{L}}$-module, $V\cong F_{a,b,c}(\lambda)$ or $V\cong \bar{F}_{a,b}(\lambda)$ for some $a,b\in\mathbb{C}$ and $c\in\mathbb{C}^*$.
\end{theo}

\subsection{\texorpdfstring{$\mathcal{A}$-}-cover of  uniformly bounded \texorpdfstring{$\bar{\mathcal{L}}$-}-modules}

Let $p,q\in\frac{1}{2}\mathbb{Z}$, $\alpha_{p},\beta_{q}\in\mathbb{Z}$, and $(\alpha+\beta)_{p+q}=\alpha_{p}+\beta_{q}$. Note that $0_0=0$.  
For any $m,i\in\mathbb{N}$ and $k,s,\alpha_{k-i},\beta_{s+i}\in\mathbb{Z}$, we define the following linear operator:
\[
\Omega_{k,s,\alpha,\beta}^{(m)} = \sum_{0\leq i\leq m} (-1)^i \binom{m}{i} L_{k-i,\alpha_{k-i}} L_{s+i,\beta_{s+i}}\in U(\bar{\mathcal{L}}).
\]

The following lemma follows from a direct computation:

\begin{lemm}\label{lem3.7}
For all $k,s,i,j,\alpha_i,\beta_i\in\mathbb Z$ and $(\alpha+\beta)_{i+j}=\alpha_i+\beta_j$, the linear operators $\Omega_{k,s,\alpha,\beta}^{(3)}$ lie in the annihilators of all intermediate series $\mathcal L$-modules $F_{a,b,c}(\lambda)$.
\end{lemm}
\begin{proof}
Let $v_p\in F_{a,b,c}(\lambda)$ with $p\in \frac12\mathbb Z$. We compute
\[
\begin{aligned}
\Omega_{k,s,\alpha,\beta}^{(3)} v_p
&= \sum_{0\leq i\leq3}(-1)^{i}\binom{3}{i}\bigl(L_{k-i,\alpha_{k-i}} L_{s+i,\beta_{s+i}}\Bigr)v_p \\
&= \sum_{0\leq i\leq3}(-1)^{i}\binom{3}{i}\bigl(i^2 A + iB + C\Bigr)\lambda^{(\alpha+\beta)_{k+s}} v_{k+s+p},
\end{aligned}
\]
where
\[
A = -b(b+1),\quad B = (k-s)b^2-2sb-a+p,\quad C=(a+sb-p)(a+kb-s-p).
\]
The summation is the third‑order finite difference of a quadratic polynomial in $i$, hence it vanishes identically. Therefore $\Omega_{k,s,\alpha,\beta}^{(3)}v_p = 0$, which completes the proof.
\end{proof}

\subsubsection{Annihilators of uniformly bounded \texorpdfstring{$\bar{\mathcal{L}}$-}-modules}
One of the principal objectives of this section is to extend the Virasoro annihilator for uniformly bounded modules, originally introduced in \cite{BF1}, to the loop mirror Heisenberg-Virasoro algebra.

To establish the theorem stated below, we first introduce some notation and a shift operator. For arbitrary pairwise distinct indices $i_1, \dots, i_d$ selected from $\{1, \dots, d\}$, where $d \in \mathbb{Z}_+$, the corresponding index permutation map is defined by
$$
	\tau_{i_1,\ldots,i_d}(x_1\cdots x_d)=x_{i_1}\cdots x_{i_d}.
$$ Denote the anticommutator $[X,Y]_+ = XY+YX$.
For $m_i,\alpha_{m_i}^i\in\mathbb{Z}$,   $i\in\{1,\ldots,d\}$,
we denote $$L(\alpha_{m_1}^1,\ldots,\alpha^d_{m_d})=L_{m_1,\alpha_{m_1}^1},
\ldots,L_{m_d,\alpha_{m_d}^d}\in U(\bar{\mathcal{L}}).$$ Given $a,b,d\in\mathbb Z_+$ satisfying $a\neq b$, we define the shift operator associated with the $d$-parameter family by
\begin{equation*} 
S_{ab}^{(d)}(L(\alpha_{m_1}^1,\ldots,\alpha^d_{m_d}))=L(\alpha_{m_1}^1,\ldots,\alpha_{m_a-1}^a,\ldots,\alpha_{m_b+1}^b,\ldots,\alpha^d_{m_d}),\quad\Delta_{ab}^{(d)}:=1-S_{ab}^{(d)}.
\end{equation*}
For simplicity, we suppress the dependence on $d$ and write $S_{ab}:=S_{ab}^{(d)}$ and $\Delta_{ab}:=\Delta_{ab}^{(d)}$. These operators act only on indices and therefore commute with one another. Consequently, every $\Delta^m_{ab}$ can be regarded as an element of the polynomial ring $\mathbb{C}[\Delta_{ab}]$.
We note that index permutation maps and shift operators do not commute. The following provides a counterexample:
$$\tau_{132}\big(S_{12}L(\alpha_k,\beta_s,\gamma_q)\big)=L(\alpha_{k-1},\gamma_q,\beta_{s+1})\neq L(\alpha_{k-1},\gamma_{q+1},\beta_{s})= S_{12}
\big(\tau_{132}L(\alpha_k,\beta_s,\gamma_q)\big).$$
Clearly,
$$
\Omega_{k,s,\alpha,\beta}^{(m)} =\sum_{i=0}^{m} (-1)^i \binom{m}{i} L(\alpha_{k-i},\beta_{s+i})
=\Delta_{12}^mL(\alpha_{k},\beta_{s}).
$$
Indeed, $\Delta_{12}^m$ is precisely the $m$-th order finite difference operator applied to the univariate function $f(i)=ai+b$. We now present an example illustrating how the difference operator $\Delta$ simplifies the computation.
\begin{exam}
\label{lem:recursion-Omega}
For all $k,s,i,\alpha_i,\beta_i\in\mathbb{Z}$ and $m\in\mathbb{N}$, the identity
\[
\Omega_{k,s,\alpha,\beta}^{(m+1)} = \Omega_{k,s,\alpha,\beta}^{(m)}-\Omega_{k-1,s+1,\alpha,\beta}^{(m)}
\]
holds.
\end{exam}
\begin{proof}
By the definition of the difference operator $\Delta_{12}$, it follows that
\[
\begin{aligned}
\Omega_{k,s,\alpha,\beta}^{(m)}-\Omega_{k-1,s+1,\alpha,\beta}^{(m)}
=\Delta_{12}^{m}(1-S_{12})L(\alpha_{k},\beta_{s})=\Delta_{12}^{m+1}L(\alpha_{k},\beta_{s})
=\Omega_{k,s,\alpha,\beta}^{(m+1)}.
\end{aligned}
\]
\end{proof}
The preceding example illustrates the recursive property satisfied by the operator $\Omega_{k,s,\alpha,\beta}^{(m)}$.

We now establish the following theorem, referred to as the loop version of the Billig-Futorny identity, which provides a simpler proof of Theorem~3.3 than that presented in \cite{BF1}.
\begin{theo}\label{thm:3.9}
Let $k,s,p,q,i,j,\alpha_i,\beta_i,\gamma_i,\eta_i$ be integers.   Suppose that $f,g\in\{\alpha,\beta,\gamma,\eta\}$  satisfy  $(f+g)_{i+j}=f_i+g_j$. 
Then the following identity holds for all integers $m,n\ge 2$:
\[
\begin{aligned}
&\sum_{\substack{0\leq i\leq m,\\ 0\leq j\leq n}}(-1)^{i+j}\binom{m}{i}\binom{n}{j}
\Biggl(\Bigl[\Omega_{k-i,s-j,\alpha,\beta}^{(m)},\Omega_{q+i,p+j,\gamma,\eta}^{(n)}\Bigr]_+
-\Bigl[\Omega_{k-i,q-j,\alpha,\gamma}^{(m)},\Omega_{s+i,p+j,\beta,\eta}^{(n)}\Bigr]_+\Biggr)
\\=&(q-s)\Biggl(
(p-k+2n)\Omega^{(2(m+n)-1)}_{k+p+2n,s+q-2n,\alpha+\eta,\beta+\gamma}
 -(p-k+2m)\Omega^{(2(m+n)-1)}_{k+p+2n-1,s+q-2n+1,\alpha+\eta,\beta+\gamma}
\Biggr).
\end{aligned}
\]
\end{theo}
\begin{proof} 
By the   definition of  $\tau$, 
we compute the left‑hand side (LHS) of the equation:
\[
\begin{aligned}
\text{LHS}
&=\bigl(\tau_{1234}+\tau_{3412}\bigr)
\sum_{\substack{0\le i,a\le m,\\ 0\le j,b\le n}}
(-1)^{i+j+a+b}
\binom{m}{a}\binom{m}{i}\binom{n}{b}\binom{n}{j}\\[4pt]
&\quad\times \biggl(L\bigl(\alpha_{k-i-a},\beta_{s-j+a},\gamma_{q+i-b},\eta_{p+j+b}\bigr)-L\bigl(\alpha_{k-i-a},\gamma_{q-j+a},\beta_{s+i-b},\eta_{p+j+b}\bigr)\biggl).
\end{aligned}
\]
Using the properties of shift operators and index permutation maps, we obtain
\[
\begin{aligned}
\text{LHS}
&=\Delta_{12}^{m}\Delta_{34}^{n}\Delta_{13}^{m}\Delta_{24}^{n}
\Bigl(\bigl(\tau_{1234}+\tau_{3412}\bigr)-\bigl(\tau_{1324}+\tau_{2413}\bigr)\Bigr)
L(\alpha_k,\beta_s,\gamma_q,\eta_p)\\[4pt]
&=\Delta_{12}^{m}\Delta_{34}^{n}\Delta_{13}^{m}\Delta_{24}^{n}
\Bigr(
\bigl(\tau_{1234}-\tau_{1324}\bigr)
+\bigl(\tau_{4132}-\tau_{4123}\bigr)
+\bigl(\tau_{3412}-\tau_{4312}\bigr)\\[4pt]
&\quad+\bigl(\tau_{4312}-\tau_{4132}\bigr)
+\bigl(\tau_{4123}-\tau_{4213}\bigr)
+\bigl(\tau_{4213}-\tau_{2413}\bigr)
\Bigr)
L(\alpha_k,\beta_s,\gamma_q,\eta_p).
\end{aligned}
\]
Examining the last four terms in the above equation, we arrive at the following straightforward results: all of them vanish.
\begin{clai} 
$\Delta_{12}^m \Delta_{34}^n \Delta_{13}^m \Delta_{24}^n(\tau_{3412}-\tau_{4312})L(\alpha_k,\beta_s,\gamma_q,\eta_p)=0.$
\end{clai}
One can readily verify that
\[
\begin
{aligned}
&\Delta_{12}^m \Delta_{34}^n \Delta_{13}^m \Delta_{24}^n(\tau_{3412}-\tau_{4312})L(\alpha_k,\beta_s,\gamma_q,\eta_p)
\\=&\Delta_{12}^m \Delta_{34}^n \Delta_{13}^m \Delta_{24}^n[L(\gamma_q),L(\eta_p)]L(\alpha_k)L(\beta_s)
\\=&\Delta_{12}^m \Delta_{13}^m \Delta_{24}^n\sum_{b=0}^n(-1)^b\binom{n}{b}[L(\gamma_{q-b}),L(\eta_{p+b})]L(\alpha_k)L(\beta_s)
\\=&\Delta_{12}^m \Delta_{13}^m \Delta_{24}^n\sum_{b=0}^n(-1)^b\binom{n}{b}(q-p-2b)L\big((\gamma+\eta)_{q+p}\big)L(\alpha_k)L(\beta_s).
\end
{aligned}
\]
By the difference annihilation lemma,
$
\sum_{i=0}^{k}(-1)^{i}\binom{k}{i}f(i)=0
$
for any integers $k\geq2$ and ${\rm deg}(f)\leq k-1$, the claim follows when $n\geq2$.

By a similar method as in the proof of the claim, we have
\[
\begin
{aligned}
&\Delta_{12}^m \Delta_{34}^n \Delta_{13}^m \Delta_{24}^n
(\tau_{4312}-\tau_{4132})L(\alpha_k,\beta_s,\gamma_q,\eta_p)
\\=&\Delta_{12}^m \Delta_{34}^n \Delta_{13}^m \Delta_{24}^n(\tau_{4123}-\tau_{4213})L(\alpha_k,\beta_s,\gamma_q,\eta_p)
\\=&\Delta_{12}^m \Delta_{34}^n \Delta_{13}^m \Delta_{24}^n(\tau_{4213}-\tau_{2413})L(\alpha_k,\beta_s,\gamma_q,\eta_p)=0.
\end
{aligned}
\]
Therefore, we can rewrite ${\rm LHS}$  as follows: 
\[
\begin
{aligned}
{\rm LHS}=&\Delta_{12}^m \Delta_{34}^n \Delta_{13}^m \Delta_{24}^n\Big((\tau_{1234}-\tau_{1324})+(\tau_{4132}-\tau_{4123})\Big)L(\alpha_k,\beta_s,\gamma_q,\eta_p)
\\=&\Delta_{12}^m \Delta_{34}^n \Delta_{13}^m \Delta_{24}^n\Big([L(\alpha_k),L(\eta_p)][L(\beta_s),L(\gamma_q)]+L(\alpha_k)[[L(\beta_s),L(\gamma_q)],L(\eta_p)]\Big)
\\=&\Delta_{12}^m \Delta_{34}^n \Delta_{13}^m \Delta_{24}^n\Bigr([L(\alpha_k),L(\eta_p)][L(\beta_s),L(\gamma_q)]\Bigr)
\\&+\Delta_{12}^m \Delta_{13}^m \Delta_{24}^n\sum_{0\leq b\leq n}(-1)^b\binom{n}{b}\Bigr(L(\alpha_k)[[L(\beta_s),L(\gamma_{q-b})],L(\eta_{p+b})]\Bigr)
\end
{aligned}
\]
\[
\begin
{aligned}
=&\Delta_{12}^m \Delta_{34}^n \Delta_{13}^m \Delta_{24}^n\Bigr([L(\alpha_k),L(\eta_p)][L(\beta_s),L(\gamma_q)]\Bigr)
\\=&\sum_{\substack{0\leq i,a\leq m,\\ 0\leq j,b\leq n}}(-1)^{i+j+a+b}\binom{m}{a}\binom{n}{b}\binom{m}{i}\binom
{n}{j} \Bigr(p-k+i+j+a+b\Bigr)
\\&
\times\Bigr(q-s+i+j-a-b\Bigr)
L\Bigr((\alpha+\eta)_{k+p-i+j-a+b},(\beta+\gamma)_{s+q+i-j+a-b}\Bigr).
\end
{aligned}
\]
By performing the variable substitutions 
$j\mapsto n-j$ and $b\mapsto n-b$, we obtain 
\[
\begin
{aligned}
{\rm LHS}=&
\sum_{\substack{0\leq i,a\leq m,\\ 0\leq j,b\leq n}}(-1)^{i+j+a+b}\binom{m}{a}\binom{n}{b}\binom{m}{i}\binom
{n}{j}\Bigr(p-k+2n+(i+a)-(j+b)\Bigr
)
\\&
\times(q-s+i-j-a+b) L\Bigr((\alpha+\eta)_{k+p+2n-i-j-a-b},(\beta+\gamma)_{s+q-2n+i+j+a+b}\Bigr)
\\=&\sum_{\substack{0\leq i,a\leq m,\\ 0\leq j,b\leq n}}(-1)^{i+j+a+b}\binom{m}{a}\binom{n}{b}\binom{m}{i}\binom
{n}{j}\Bigr(p-k+2n+(i+a)-(j+b)\Bigr
)
\\&
\times\Bigr(q-s+i-j-a+b\Bigr) S_{12}^{i+j+a+b}L\Bigr((\alpha+\eta)_{k+p+2n},(\beta+\gamma)_{s+q-2n}\Bigr).
\end{aligned}
\]
We now focus on the coefficient of  $L\Bigr((\alpha+\eta)_{k+p+2n},(\beta+\gamma)_{s+q-2n}\Bigr)$  in the above equation, namely,
\[
\begin
{aligned}
&\sum_{\substack{0\leq i,a\leq m,\\ 0\leq j,b\leq n}}(-1)^{i+j+a+b}\binom{m}{a}\binom{n}{b}\binom{m}{i}\binom
{n}{j}\Bigl(p-k+2n+(i+a)-(j+b)\Bigr
)
\\&
\times\Bigr(q-s+i-a-j+b\Bigr) S_{12}^{i+j+a+b}
\\=&(q-s)\Bigl((p-k+2n)(1-S_{12})^{2(m+n)}
  +2(1-S_{12})^{2n+m}S_{12}\frac{d}{dS_{12}}(1-S_{12})^{m}
  \\&-2(1-S_{12})^{2m+n}S_{12}\frac{d}{dS_{12}}(1-S_{12})^{n} \Bigl)\\
=
  &(q-s)\Bigl((p-k+2n)(1-S_{12})^{2(m+n)}-2mS_{12}(1-S_{12})^{2(m+n)-1}+2nS_{12}(1-S_{12})^{2(m+n)-1} \Bigl)
  \\
=
 &(q-s)\Bigl((p-k+2n)\Delta^{2(m+n)-1}-(p-k+2m)S_{12}\Delta^{2(m+n)-1}\Bigl).
\end{aligned}
\]
It is straightforward to verify that
\[
\begin
{aligned}
{\rm LHS}=&
(q-s)\Bigl((p-k+2n)\Delta^{2(m+n)-1}-(p-k+2m)S_{12}\Delta^{2(m+n)-1}\Bigl)
\\&\times L\Bigr((\alpha+\eta)_{k+p+2n},(\beta+\gamma)_{s+q-2n}\Bigr)
={\rm RHS}.
\end{aligned}
\]
This completes the proof.
\end{proof}

For brevity, write $H(\alpha_r)=H_{r,\alpha_r}$ for any $r\in\mathbb{Z}+\frac{1}{2}$ and $\alpha_r\in\mathbb{Z}$.
\begin{lemm}\label{lemm3.10}
Let $V$ be a uniformly bounded $\bar{\mathcal{L}}$-module with a composition series of length $h\in\mathbb{Z}_+$. Then, for every $h$,
\begin{enumerate}
    \item[\rm (i)] there exists $l\in \mathbb{N}$ such that, for all
$k,s,\alpha_i,\beta_i\in\Z$ satisfying $(\alpha+\beta)_{i+j}=\alpha_{i}+\beta_{j}$, the operator $\Omega_{k,s,\alpha,\beta}^{(l)}$ annihilates $V$;
    \item[\rm (ii)] there exists $l\in \mathbb{N}$ such that, for all
$r\in\mathbb{Z}+\frac{1}{2},k,\gamma_r,\beta_i\in\Z$ satisfying $(\gamma+\beta)_{r+i}=\gamma_{r}+\beta_{i}$, the operator $$\Psi_{r,k,\alpha,\beta}^{(l)}=\Delta_{12}^{l}\big(H(\gamma_{r})L(\beta_{k})\big)\in U(\bar{\mathcal{L}})$$ annihilates $V$.
\end{enumerate}
\end{lemm}
\begin{proof}
(i) Let $V$ be a uniformly bounded $\bar{\mathcal {L}}$-module with a composition series of length $h$:
\[
0=V_0\subset V_1\subset V_2\subset\cdots\subset V_h=V,
\]
where each quotient $V_{i+1}/V_i$ is simple. We proceed by induction on the composition length $h$.

Nontrivial highest weight or lowest weight $\bar{\mathcal {L}}$-modules cannot be uniformly bounded. Therefore, the simple composition factors of $V$ can only be of the form $\bar{F}_{a,b}(\lambda)$.

Base case: $h=1$. Then $V$ itself is a uniformly bounded simple module. By Lemma \ref{lem3.7}, $\Omega_{k,s,\alpha,\beta}^{(3)}$ annihilates every uniformly bounded simple module, so statement (i) holds for $h=1$.

Inductive step: Suppose that $h>1$. Write $h=h_1+h_2$, where $h_1,h_2\ge 1$ are integers. By the induction hypothesis, there exist positive integers $m,n\in\mathbb Z_+$ such that $\Omega_{k,s,\alpha,\beta}^{(m)}$ annihilates all uniformly bounded modules of length $h_1$, and $\Omega_{p,q,\gamma,\eta}^{(n)}$ annihilates all uniformly bounded modules of length $h_2$.

Consider the products $\Omega_{k,s,\alpha,\beta}^{(m)}\Omega_{p,q,\gamma,\eta}^{(n)}$ and $\Omega_{p,q,\gamma,\eta}^{(n)}\Omega_{k,s,\alpha,\beta}^{(m)}$. Both products annihilate $V$. Indeed,
\[
\Omega_{p,q,\gamma,\eta}^{(n)} V\subset V_{h_1}.
\]
Since $\Omega_{p,q,\gamma,\eta}^{(n)}\in\operatorname{Ann}(V/V_{h_1})$, it maps the entire module $V$ into the submodule $V_{h_1}$. The quotient $V/V_{h_1}$ has composition series length $h_2$, whereas $V_{h_1}$ has length $h_1$. As $\Omega_{k,s,\alpha,\beta}^{(m)}$ annihilates $V_{h_1}$, it follows that
\[
\Omega_{k,s,\alpha,\beta}^{(m)}\Omega_{p,q,\gamma,\eta}^{(n)} V\subset \Omega_{k,s,\alpha,\beta}^{(m)} V_{h_1}=0.
\]
By symmetry, $\Omega_{p,q,\gamma,\eta}^{(n)}\Omega_{k,s,\alpha,\beta}^{(m)}V=0$. 

If $m=n$, Theorem \ref{thm:3.9} and the recursive property of $\Omega_{k,s,\alpha,\beta}^{(m)}$ imply that $\Omega_{k,s,\alpha+\gamma,\beta+\eta}^{(4m)}\in\operatorname{Ann}(V)$. If $m\neq n$, then
\[
(p-k+2n)\Omega_{k+p+2n,s+q-2n}^{(2m+2n-1)}-(p-k+2m)\Omega_{k+p+2n-1,s+q-2n+1}^{(2m+2n-1)}
\]
annihilates $V$ for all $p,k,s,q$. We may vary $p-k=-2m$ while keeping $p+k$ fixed. This implies that $\Omega_{k+p+2n,s+q-2n}^{(2m+2n-1)}\in\operatorname{Ann}(V)$, completing the proof of (i).

(ii) Let $m\geq2$.  From ${\rm (i)}$,  we immediately get $[H(\gamma_r),\Omega_{k,p,\alpha,\beta}^{(m)}]V=0$  for all $k,p,\alpha_i,\beta_i,\gamma_r\in\mathbb{Z}, r\in \mathbb{Z}+\frac{1}{2}$. Thus, on $V$ we have
\begin{align*}
0=&\sum_{\substack{0\leq i\leq 2,\\ 0\leq a\leq 1}}(-1)^{i+a}\binom{2}{i}\binom{1}{a}\Bigl([H(\gamma_{\frac{1}{2}-i-a}),\Omega_{k+a,p+i,\alpha,\beta}^{(m)}]\Bigr)V
\\=& \sum_{\substack{0\leq i\leq 2,\\ 0\leq a\leq 1,\\ 0\leq j\leq m}}(-1)^{i+a+j}\binom{2}{i}\binom{1}{a}\binom{m}{j}\Bigl([H(\gamma_{\frac{1}{2}-i-a}),L(\alpha_{k+a-j},\beta_{p+i+j})]\Bigr)V
\\=& \sum_{\substack{0\leq i\leq 2,\\ 0\leq a\leq 1,\\ 0\leq j\leq m}}(-1)^{i+a+j}\binom{2}{i}\binom{1}{a}\binom{m}{j}\Bigl((\frac{1}{2}-i-a)H\big((\alpha+\gamma)_{\frac{1}{2}+k-i-j}\big)L(\beta_{p+i+j})
\\&+(\frac{1}{2}-i-a)L(\alpha_{k+a-j})H\big((\beta+\gamma)_{p+j+\frac{1}{2}-a}\big)\Bigr)V
\\=& \sum_{\substack{0\leq i\leq 2,\\ 0\leq a\leq 1}}(-1)^{i+a}\binom{2}{i}\binom{1}{a}\Bigl((\frac{1}{2}-i-a)S_{12}^{i}\Delta_{12}^mH\big((\alpha+\gamma)_{\frac{1}{2}+k}\big)L(\beta_{p})
\\&+(\frac{1}{2}-i-a)\Delta_{12}^mL(\alpha_{k+a})H\big((\beta+\gamma)_{p+\frac{1}{2}-a}\big)\Bigr)V
\\=&\Delta_{12}^{m+2}H\big((\alpha+\gamma)_{\frac{1}{2}+k}\big)L(\beta_{p})V+\sum_{0\leq i\leq 2}(-1)^{i}\binom{2}{i}(\frac{1}{2}-i)\Delta_{12}^mL(\alpha_{k})H\big((\beta+\gamma)_{p+\frac{1}{2}}\big)V
\\&+\sum_{0\leq i\leq 2}(-1)^{i}\binom{2}{i}(\frac{1}{2}+i)\Delta_{12}^mL(\alpha_{k+1})H\big((\beta+\gamma)_{p-\frac{1}{2}}\big)V
=\Psi_{k+\frac{1}{2},p,\alpha+\gamma,\beta}^{(m+2)}V.
\end{align*}
\end{proof}

\subsubsection{Classification of uniformly bounded \texorpdfstring{$\bar{\mathcal{L}}$-}-modules}
We recall the definitions of coinduced modules and $\mathcal{A}$‑covers (see \cite{BF1}).

\begin{defi}
A coinduced module from an $\bar{\mathcal{L}}$‑module $V$ is the space $\mathrm{Hom}(\mathcal{A},V)$, equipped with actions of $\bar{\mathcal{L}}$ and $\mathcal{A}$ defined as follows:
\begin{align*}
(\mathfrak{a}\varphi)(f)=\mathfrak{a}\bigl(\varphi(f)\bigr)-\varphi\bigl(\mathfrak{a}(f)\bigr),
\qquad
(y\varphi)(f)=\varphi(yf),
\end{align*}
where $\varphi\in\mathrm{Hom}(\mathcal{A},V)$, $\mathfrak{a}\in \bar{\mathcal{L}}$, and $f,y\in \mathcal{A}$.
\end{defi}

\begin{defi}\label{def5777}
An $\mathcal{A}$‑cover of a uniformly bounded $\bar{\mathcal{L}}$‑module $V$ is the $\mathcal{A}\bar{\mathcal{L}}$‑submodule
\[
\widehat{V}=\mathrm{span}\Bigl\{\phi(\mathfrak{a},w) \,\Big\vert\, \mathfrak{a}\in\bar{\mathcal{L}},\;w\in V\Bigr\}\subset\mathrm{Hom}(\mathcal{A},V),
\]
where the map $\phi(\mathfrak{a},w)\colon \mathcal{A}\to V$ is defined by
\[
\phi(\mathfrak{a},w)(f)=(f\mathfrak{a})(w).
\]
\end{defi}
The $\mathcal{A}\bar{\mathcal{L}}$‑action on $\widehat{V}$ is defined by
\begin{align*}
\mathfrak{b}\phi(\mathfrak{a},w)&=\phi\bigl([\mathfrak{b},\mathfrak{a}],w\bigr)+\phi(\mathfrak{a},\mathfrak{b}w),\\
f\phi(\mathfrak{a},w)&=\phi(f\mathfrak{a},w)
\end{align*}
for all $\mathfrak{a},\mathfrak{b}\in \bar{\mathcal{L}}$, $w\in V$, and $f\in \mathcal{A}$.

Define
\[
\mathcal{K}(V)=\left\{\sum_{r\in\frac{1}{2}\mathbb Z}\mathfrak{a}_r\otimes w_{r}\in\bar{\mathcal{L}}\otimes V
\,\bigg\vert\, \sum_{r\in\frac{1}{2}\mathbb Z}(f\mathfrak{a}_r) w_{r}=0\quad \forall\, f\in\mathcal{A} \right\}.
\]
Then $\mathcal{K}(V)$ is an $\mathcal{A}\bar{\mathcal{L}}$‑submodule of $\bar{\mathcal{L}}\otimes V$. Under the assumption $\bar{\mathcal{L}}V=V$, the $\mathcal{A}$‑cover $\widehat{V}$ can also be realized as the quotient $\mathcal{A}\bar{\mathcal{L}}$‑module
\[
(\bar{\mathcal{L}}\otimes V)/\mathcal{K}(V).
\]
Clearly, the linear map
\begin{align*}
\sigma\colon\quad \widehat{V} &\longrightarrow \bar{\mathcal{L}}V\\
\mathfrak{a}\otimes w+\mathcal{K}(V) &\longmapsto \mathfrak{a}w
\end{align*}
is an $\bar{\mathcal{L}}$‑module epimorphism.

\begin{lemm}\label{lem:3.11}
Let $V$ be a uniformly bounded module over $\bar{\mathcal{L}}$. Then the $\A$-cover $\widehat{V}$ of $V$ is also  a uniformly bounded $\A\bar{\mathcal{L}}$-module.
\end{lemm}
\begin{proof}
The assertion is evident if $V$ is trivial. We therefore assume that $V$ is nontrivial. Under this assumption, $\bar{\mathcal{L}}V = V$, and the support satisfies $\operatorname{Supp}(V)\subseteq \lambda+\tfrac12\mathbb{Z}$. Suppose further that $\dim V_\theta\le n$ for every weight $\theta\in\operatorname{Supp}(V)$. By Lemma~\ref{lemm3.10}, there exists a positive integer $m\in\mathbb{N}$ such that
\[
\Delta_{12}^{m}L(\alpha_k,\beta_p)v
=\Delta_{12}^{m}H(\alpha_{k+\frac12})L(\beta_p)v=0
\]
for all $k,p\in\mathbb{Z}$ and all $v\in V$. Consequently,
\begin{equation}\label{eq:3.4}
\Delta_{12}^{m}L(\alpha_k,\beta_p)v
, \quad \Delta_{12}^{m}H(\alpha_{k+\frac12})L(\beta_p)v\in\mathcal{K}(V).
\end{equation}

Define the finite dimensional subspace
\[
R=\operatorname{span}\bigl\{L(\alpha_i),\,H(\beta_{i+\frac12})\,\big|\,0\le i\le m\bigr\}.
\]
Then $\dim R = 4(m+1)$. Moreover, $R\otimes V$ is a $\mathbb{C}L_0$-submodule of $\bar{\mathcal{L}}\otimes V$, and its weight space dimensions satisfy
\[
\dim(R\otimes V)_\theta \le 4(m+1)n,\qquad \forall\ \theta\in\lambda+\tfrac12\mathbb{Z}.
\]
Our goal is to establish the decomposition
\[
\bar{\mathcal{L}}\otimes V = R\otimes V+\mathcal{K}(V),
\]
which immediately implies that the associated covering module $\widehat{V}$ is uniformly bounded.

To this end, we proceed by induction on $n$ to show that, for any weight vector $u\in V_\theta$,
\[
L(\alpha_n)\otimes u,\qquad H(\beta_{n+\frac12})\otimes u \;\in R\otimes V.
\]
We prove only the case $n>m$, as the proof for $n<0$ is similar.
Since $L_0$ acts on the weight space $V_\theta$ as multiplication by a nonzero scalar, we may write $u = L_0v$ for some $v\in V_\theta$. Applying relation~\eqref{eq:3.4} together with the inductive hypothesis, we obtain
\begin{align*}
L(\alpha_n)\otimes L_0v
&=\Delta_{12}^{m}\big(L(\alpha_{n-i},0_i)\big)v
-\sum_{i=1}^{m}(-1)^{i}\binom{m}{i}L(\alpha_{n-i},0_i)v,\\[4pt]
H\big(\beta_{n+\frac12}\big)\otimes L_0v
&=\Delta_{12}^{m}\big(H(\beta_{n-i+\frac12})L(0_i)\big)v
-\sum_{i=1}^{m}(-1)^{i}\binom{m}{i}H(\alpha_{n-i+\frac12})L(0_i)v.
\end{align*}
Both expressions belong to $R\otimes V+\mathcal{K}(V)$. This completes the inductive step and proves the lemma.
\end{proof}

\begin{remark}
By Lemma~\ref{lem:3.11}, uniformly bounded $\bar{\mathcal{L}}$‑modules and uniformly bounded $\mathcal{A}\bar{\mathcal{L}}$‑modules are identical as $\bar{\mathcal{L}}$‑modules, disregarding the extra associative algebra action of $\mathcal{A}$.
\end{remark}

 Now we give a classification for all  uniformly bounded simple $\bar{\mathcal{L}}$-modules.
\begin{theo}\label{thm:3.13}
Any  uniformly bounded simple $\bar{\mathcal{L}}$-module is isomorphic to a module of intermediate series $V\cong F_{a,b,c}(\lambda)$ or $V\cong \bar{F}_{a,b}(\lambda)$ for some $a,b\in\mathbb{C}$, $c\in\mathbb{C}^*$.
\end{theo}
\begin{proof}
Let $V$ be a nontrivial uniformly bounded simple $\bar{\mathcal{L}}$-module. Then $\bar{\mathcal{L}}V=V$, and there exists an epimorphism $\pi\colon \widehat{V}\to V$. By Lemma~\ref{lem:3.11}, $\widehat{V}$ is uniformly bounded. Hence, $\widehat{V}$ admits a composition series of $\mathcal{A}\bar{\mathcal{L}}$-submodules:
\[
0=\widehat{V}^{(0)}\subset \widehat{V}^{(1)}\subset\cdots\subset \widehat{V}^{(d)}=\widehat{V}
\]
such that each $\widehat{V}^{(i)}/\widehat{V}^{(i-1)}$ is a simple $\mathcal{A}\bar{\mathcal{L}}$-module. Let $g$ be the minimal integer such that $\pi(\widehat{V}^{(g)})\neq 0$ and $\pi(\widehat{V}^{(g-1)})=0$. Since $V$ is simple, it follows that $\pi(\widehat{V}^{(g)})=V$. Therefore, there exists an $\bar{\mathcal{L}}$-epimorphism from the simple $\mathcal{A}\bar{\mathcal{L}}$-module $\widehat{V}^{(g)}/\widehat{V}^{(g-1)}$ onto $V$. The result then follows from Theorem~\ref{thm:3.4}.
\end{proof}

\subsection{Classification theorem of uniformly bounded
  \texorpdfstring{$\mathcal{L}$}{L}-modules}
  Based on Theorem \ref{thm:2.3}, the action of $C_{1,i}$ on any uniformly bounded simple module is trivial. 
  \begin{lemm}
Let $V$ be a uniformly bounded $\mathcal{L}$-module. Then, for all $i\in\mathbb{Z}$, we have $C_{2,i}V=0$.
\end{lemm}
\begin{proof}
Assume $m\ge 4$. It follows from Lemma \ref{lemm3.10} (ii) that
$
\bigl[H(\gamma_r),\Psi_{s,p,\alpha,\beta}^{(m)}\bigr]$
annihilates $V$ for all $p,\alpha_r,\beta_i,\gamma_r\in\mathbb Z$ and $r,s\in\mathbb Z+\frac12$. Choosing $r+s\leq m+2$, we perform a computation analogous to that presented in (ii):
\begin{align*}
0=&\sum_{\substack{0\leq i\leq 2,\\ 0\leq a\leq 1}}(-1)^{i+a}\binom{2}{i}\binom{1}{a}\Bigl([H(\gamma_{r-i-a}),\Psi_{s+a,p+i,\alpha,\beta}^{(m)}]\Bigr)V
\\=& \sum_{\substack{0\leq i\leq 2,\\ 0\leq a\leq 1,\\ 0\leq j\leq m}}(-1)^{i+a+j}\binom{2}{i}\binom{1}{a}\binom{m}{j}\Bigl([H(\gamma_{r-i-a}),H(\alpha_{s+a-j})L(\beta_{p+i+j})]\Bigr)V
\\=& \sum_{\substack{0\leq i\leq 2,\\ 0\leq a\leq 1,\\ 0\leq j\leq m}}(-1)^{i+a+j}\binom{2}{i}\binom{1}{a}\binom{m}{j}\Bigl((r-i-a)\delta_{r+s-i-j,0}C_{2,(\alpha+\gamma)_{r+s-i-j}}L(\beta_{p+i+j})
\\&+(r-i-a)H(\alpha_{s+a-j})H\big((\beta+\gamma)_{p+j+r-a}\big)\Bigr)V
\\=& \sum_{\substack{0\leq i\leq 2,\\ 0\leq a\leq 1}}(-1)^{i+a}\binom{2}{i}\binom{1}{a}\Bigl((r-i-a)S_{12}^{i}\Delta_{12}^m\delta_{r+s,0}C_{2,(\alpha+\gamma)_{r+s}}L(\beta_{p})
\\&+(r-i-a)\Delta_{12}^mH(\alpha_{s+a})H\big((\beta+\gamma)_{p+r-a}\big)\Bigr)V
\end{align*}
\begin{align*}
=&\Delta_{12}^{m+2}\big(\delta_{r+s,0}C_{2,(\alpha+\gamma)_{r+s}}L(\beta_{p})\big)V+\sum_{0\leq i\leq 2}(-1)^{i}\binom{2}{i}(r-i)\Delta_{12}^mH(\alpha_{s})H\big((\beta+\gamma)_{p+r}\big)V
\\&-\sum_{0\leq i\leq 2}(-1)^{i}\binom{2}{i}(r-i-1)\Delta_{12}^m L(\alpha_{s+1})H\big((\beta+\gamma)_{p+r-1}\big)V
\\=&\sum_{0\leq i\leq 2}(-1)^{i}\binom{m+2}{i}\big(\delta_{r+s-i,0}C_{2,(\alpha+\gamma)_{r+s-i}}L(\beta_{p+i})\big)V
\\=&(-1)^{r+s}\binom{m+2}{r+s}C_{2,(\alpha+\gamma)_{0}}L(\beta_{p+r+s})V.
\end{align*}
Setting $\beta_{p+r+s}=p+r+s=0$ in the preceding equation, we conclude that $C_{2,i}$ acts trivially on $V$ for all $i\in\mathbb{Z}$.
\end{proof}
\begin{remark}
    The category of uniformly bounded simple $\mathcal{L}$-modules is equivalent to the category of uniformly bounded simple $\bar{\mathcal{L}}$-modules.
\end{remark}
By Theorem \ref{thm:3.13}, we obtain the following result.
\begin{theo}\label{thm:3.16}
Let $V$ be a uniformly bounded simple $\mathcal{L}$-module. Then $V\cong F_{a,b,c}(\lambda)$ or $V\cong \bar{F}_{a,b}(\lambda)$ for some $a,b\in\mathbb{C}$ and $c\in\mathbb{C}^*$.
\end{theo}

\begin{coro}\label{cor:3.14}
Let $n\ge 2$ be an integer, and let $V$ be a simple Harish-Chandra module over the truncated mirror Heisenberg-Virasoro algebra $\mathcal{L}(n)$. Then $V$ is either a highest weight module, a lowest weight module, or a module of the form $\bar{F}_{a,b}$ for the mirror Heisenberg-Virasoro algebra, satisfying $(\hat{\mathcal{L}}\otimes t\mathbb{C}[t])V=0$.
\end{coro}

 \section{Classiﬁcation of simple Harish-Chandra modules}
Let $\mathfrak{g}$ be a $\frac12\mathbb{Z}$‑graded Lie algebra, and let $V$ be an indecomposable weight $\mathfrak{g}$‑module with weight space decomposition
\[
V=\bigoplus_{r\in\frac12\mathbb{Z}}V_{\lambda+r}.
\]
\begin{defi}
Let $V$ be a weight module.
We call $V$ \textit{upper bounded} {\rm(}resp.\ \textit{lower bounded}{\rm)} if there exist $p\in\mathbb{N}$ and $r_0\in\frac12\mathbb{Z}$ satisfying $\dim V_{\lambda+r}\leq p$ for all $r\geq r_0$ {\rm(}resp.\ $r\leq r_0${\rm)}.
We say $V$ has \textit{upper bounded weights} {\rm(}resp.\ \textit{lower bounded weights}{\rm)}  if $V_{\lambda+r}=0$ for all $r\geq r_0$ {\rm(}resp.\ $r\leq r_0${\rm)}  for some $r_0\in\frac12\mathbb{Z}$.
\end{defi}
The following result is an immediate consequence of \cite[Lemma 1.6]{M2}.
\begin{lemm}\label{lem:4.1}
Let $V$ be a Harish-Chandra module over the mirror Heisenberg-Virasoro algebra $\hat{\mathcal{L}}$ such that $\operatorname{supp}(V)\subseteq \lambda+\tfrac12\mathbb{Z}$. Suppose that, for each $v\in V$, there exists $T(v)\in\mathbb{N}$ such that
\[
L_m v = H_{m+\frac12}v = 0 \quad\text{for all } m\ge T(v).
\]
Then $\operatorname{supp}(V)$ is upper bounded.
\end{lemm}

\begin{lemm}\label{lem:4.2}
Let $V$ be a Harish-Chandra module over the mirror Heisenberg-Virasoro algebra $\hat{\mathcal{L}}$.
If $V$ is not lower bounded {\rm(}resp.\ upper bounded{\rm)}, then $V$ contains a highest weight {\rm(}resp.\ lowest weight{\rm)} $\hat{\mathcal{L}}$-submodule.
\end{lemm}

\begin{proof}
Suppose that $V$ is not lower bounded. Fix $\lambda\in\operatorname{supp}(V)$. Then there exists $r\in\tfrac12\mathbb{Z}$ such that
\[
\dim V_{\lambda-r}> 2\bigl(\dim V_{\lambda}+\dim V_{\lambda+\frac12}+\dim V_{\lambda+1}\bigr).
\]
Without loss of generality, we may assume that $r=n\in\mathbb{N}$. Then there exists a nonzero element $u\in V_{-n+\lambda}$ satisfying
\[
L_n u = L_{n+1}u = H_{n+\frac12}u = 0.
\]
It follows from $[\hat{\mathcal{L}}_r,\hat{\mathcal{L}}_s]\subset\hat{\mathcal{L}}_{r+s}$ that $L_k u = H_{k+\frac12}u = 0$ for all $k\ge n^2+n+1$.

Let $V'=\{v\in V\mid \dim\hat{\mathcal{L}}_+v<\infty\}$. It is straightforward to verify that $V'$ is a nonzero $\hat{\mathcal{L}}$-submodule of $V$. By Lemma~\ref{lem:4.1}, $\operatorname{supp}(V')$ is upper bounded. Consequently, $V$ contains a highest weight submodule.
\end{proof}
     
Let $V$ be a weight $\hat{\mathcal{L}}$-module. Denote by $V^+$ (resp. $V^{-}$) the maximal $\hat{\mathcal{L}}$-submodule of $V$ whose weights are upper bounded (resp. lower bounded).
\begin{lemm}\label{lem:4.3}
Let $V$ be a Harish-Chandra module over $\hat{\mathcal{L}}$.
\begin{enumerate}
    \item[\rm (i)]
  If $V$ is a lower bounded $\hat{\mathcal{L}}$-module, then $V^+$ {\rm(}resp.\ $V^-${\rm)} is a finite extension of trivial $\hat{\mathcal{L}}$-modules. If $V$ is not a lower bounded $\hat{\mathcal{L}}$-module, then $V^+$ {\rm(}resp.\ $V^-${\rm)} is not lower bounded {\rm(}resp.\ upper bounded{\rm)}.

    \item[\rm (ii)] If $V$ is indecomposable, then $V/V^+$ {\rm(}resp.\ $V/V^-${\rm)} is a lower bounded {\rm(}resp.\ upper bounded {\rm)} $\hat{\mathcal{L}}$-module.
\end{enumerate}
\end{lemm}

\begin{proof}
(i) Suppose that $V$ is lower bounded. Then $V^+$ is a uniformly bounded module. By \cite[Theorem III.8]{S0}, $V^+$ is a finite length extension of trivial modules over the Virasoro algebra, and hence a finite length extension of the trivial $\hat{\mathcal{L}}$-module. If $V$ is not lower bounded, then Lemma~\ref{lem:4.2} implies that $V$ contains a highest weight $\hat{\mathcal{L}}$-submodule, which must be contained in $V^+$. Consequently, $V^+$ is not lower bounded.

(ii) Assume, for contradiction, that $V/V^+$ is not lower bounded. By Lemma~\ref{lem:4.2}, $V/V^+$ contains a highest weight $\hat{\mathcal{L}}$-submodule. This contradicts the definition of $V^+$.
\end{proof}
By Lemma \ref{lem:4.3} and an argument analogous to that in \cite[Theorem 3.1, Proposition 3.3]{S0}, we obtain the following result.
\begin{lemm}\label{lem:4.4}
Let $V$ be a Harish-Chandra module over $\hat{\mathcal{L}}$ that does not contain any trivial $\hat{\mathcal{L}}$-submodules. Then there exists a decomposition of $\hat{\mathcal{L}}$-submodules
\[
V\cong V^+\oplus V^0\oplus V^-,
\]
where $V^\pm$ are defined as above and $V^0$ is a uniformly bounded module.
\end{lemm}

\subsection{Main results}
\begin{theo}\label{thm:4.5}
Let $V$ be a simple Harish-Chandra $\mathcal{L}$-module. Then $V$ is either a highest weight module, a lowest weight module, or a uniformly bounded module.
    
\end{theo} 
\begin{proof}
Suppose that $V$ is not a uniformly bounded module.
We view $V$ as an $\hat{\mathcal{L}}$-module.
Let $P$ be a minimal $\hat{\mathcal{L}}$-submodule of $V$ such that $V/P$ is trivial,
and let $T$ stand for the largest trivial $\hat{\mathcal{L}}$-submodule of $P$. Note that $T$
is finite dimensional.

By Lemma~\ref{lem:4.4}, we have the following decomposition of $\hat{\mathcal{L}}$-modules:
\[
\bar{P}:= P/(P\cap T)\cong \bar{P}^+\oplus \bar{P}^0\oplus \bar{P}^-,
\]
where $\operatorname{supp}(\bar{P}^+)$ is upper bounded, $\operatorname{supp}(\bar{P}^-)$ is lower bounded,
and $\bar{P}^0$ is uniformly bounded. Without loss of generality, we assume that $\bar{P}^+$ is nontrivial.

For any $w\in P$, denote by $\bar{w}$ its image in $\bar{P}$.
Since the central element $C_1$ acts on $V$ via scalar multiplication, we split the proof into two cases according to this scalar.

\begin{case}
The action of $C_1$ is zero.
\end{case}

Let $\omega$ be the highest weight of $\bar{P}^+$. Choose a weight vector $v_1\in P_{\omega}^+$ such that $\bar{v}_1\neq 0$.
We distinguish two subcases according to whether $\omega$ vanishes.
If $\omega\neq 0$, set $\lambda=\omega$ and $v_0=v_1$.
If $\omega=0$, consider the cyclic module $U(\hat{\mathcal{L}})v_1$.
By the minimality of $P$, the quotient
\[
U(\hat{\mathcal{L}})v_1\big/\bigl(U(\hat{\mathcal{L}})v_1\cap T\bigr)
\]
is a highest weight $\hat{\mathcal{L}}$-submodule of $\bar{P}^+$.
As vector spaces,
\[
U(\hat{\mathcal{L}})v_1 = U(\hat{\mathcal{L}}_-)v_1 + \bigl(U(\hat{\mathcal{L}})v_1\cap T\bigr).
\]
One can pick a weight vector $v_0\in U(\hat{\mathcal{L}})v_1$ of weight $-1$, i.e., $L_0 v_0 = -v_0$,
such that its image $\bar{v}_0$ is a highest weight vector in the above quotient, meaning
$\hat{\mathcal{L}}_+\bar{v}_0=0$.
In this subcase we set $\lambda=-1$.

In both subcases we have $\lambda\neq 0$.
Define $F:=U(\hat{\mathcal{L}})v_0$ and $T':=F\cap T$.
Then $F/T'$ is a nontrivial highest weight $\hat{\mathcal{L}}$-submodule of $\bar{P}^+$, and
\[
F = U(\hat{\mathcal{L}}_-)v_0 + T'
\]
as vector spaces.
Let $F'$ be the largest $\hat{\mathcal{L}}$-submodule of $F$ subject to
\[
F'_{\lambda}=F'_{\lambda-\frac12}=0.
\]
Observe $T'\subseteq F'$, and $F/F'$ is isomorphic to the nontrivial simple $\hat{\mathcal{L}}$-module $F(0,c_2,\lambda)$.

By our hypothesis, there exists $s\in \frac12\mathbb{Z}_+$ such that
\[
\dim F(0,c_2,\lambda)_{\lambda-s}>\dim V_{\lambda}+\dim V_{\lambda-\frac12}.
\]
For each $i\in\mathbb{Z}$, consider the linear map
\[
L_{s,i}\oplus H_{s-\frac12,i}\colon F''_{\lambda-s}\to V_{\lambda}\oplus V_{\lambda-\frac12},
\]
where $F''_{\lambda-s}$ is a vector space complement of $F'_{\lambda-s}$ inside $F_{\lambda-s}$.
From the dimension inequality,
\[
\dim F''_{\lambda-s}=\dim F(0,c_2,\lambda)_{\lambda-s}>\dim\bigl(V_{\lambda}\oplus V_{\lambda-\frac12}\bigr),
\]
there exists a nonzero element $w_k\in F''_{\lambda-s}\subseteq F_{\lambda-s}\setminus F'_{\lambda-s}$ satisfying
\[
L_{s,i}w_k=H_{s-\frac12,i}w_k=0.
\]

Pick $s_0\in\mathbb{N}$ so that $\lambda+j\neq 0$ and $\lambda+j-\frac12\neq 0$ for all $j\in\mathbb{N}$ with $j>s_0$.
For $j>s_0$, we have
\[
(F'/T')_{\lambda+j}=(F'/T')_{\lambda+j-\frac12}=
(F/F')_{\lambda+j}=(F/F')_{\lambda+j-\frac12}=0.
\]
Since $L_{s+j}w_k\in F_{\lambda+j}$ and $H_{s+j-\frac12}w_k\in F_{\lambda+j-\frac12}$,
it follows that $F_{\lambda+j}=F_{\lambda+j-\frac12}=0$, hence
\[
L_{s+j}w_k=H_{s+j-\frac12}w_k=0,\ \forall\,j>s_0,\ k\in\mathbb{Z}.
\]
Applying $L_{s+j,i}$ yields
\begin{align*}
&L_{s+j}(L_{s,i}w_k) = jL_{2s+j,i}w_k=0,
\\&L_{s+j}(H_{s-\frac12,i}w_k) = -\bigl(s-\frac12\bigr)\bigl(H_{2s+j-\frac12,i}w_k\bigr)=0,
\quad \forall\,j>s_0.
\end{align*}
Because $w_k\in F_{\lambda-s}\setminus F'_{\lambda-s}$, one can find elements
$x_{k_1},\dots,x_{k_t}\in\hat{\mathcal{L}}$ with $k_1,\ldots,k_t\in\frac{1}{2}\mathbb{Z}_+$
and $k_1+\cdots+k_t=s$, such that
\[
0\neq x_{k_1}\cdots x_{k_t}\bar{w}_k\in (F/F')_{\lambda}=F_{\lambda}/F'_{\lambda}.
\]
Here each $x_{k}$ equals $L_{k}$ if $k\in\mathbb{Z}$, or $H_{k}$ if $k\in\mathbb{Z}+\frac12$, where $k\in\{k_1,\ldots,k_t\}$.
Recall $F'_{\lambda}=0$ and $\lambda\neq 0$, so
\[
0\neq x_{k_1}\cdots x_{k_t}w_k\in F_{\lambda}=\mathbb{C}v_0.
\]
Thus there exists $d\in\mathbb{C}^\ast$ with $d\,x_{k_1}\cdots x_{k_t}w_k = v_0$.
Combined with $L_{2s+j,i}w_k=0$ for all $j>s_0$, we obtain
\[
L_{2s+j,i}v_0
=L_{2s+j,i}\bigl(d\,x_{k_1}\cdots x_{k_t}w_k\bigr)=0,\ \forall\,j>s_0,\ i\in\mathbb{Z}.
\]
Similarly, since $H_{2s+j-\frac12,i}w_k=0$ for $j>s_0$, we have
\[
H_{2s+j-\frac12,i}v_0
=H_{2s+j-\frac12,i}\bigl(d\,x_{k_1}\cdots x_{k_t}w_k\bigr)=0,\ \forall\,j>s_0,\ i\in\mathbb{Z}.
\]

For $N\in\mathbb{N}$, let $\mathcal{L}_{>2s+N}={\rm span}\{L_m,H_{m-\frac{1}{2}}\mid m>2s+N\}$. Then, $\mathcal{L}_{>2k+N}v_0=0$.
Now define the subspace
\[
V^+=\Bigl\{v\in V\mid \exists\,N(v)\in\mathbb{Z}_+ \  \text{s.t.}\
L_{n,i}v=H_{n-\frac12,i}v=0,\ \forall\,n>N(v),\ i\in\mathbb{Z}\Bigr\}.
\]
We have shown that $0\neq v_0\in V^+$. It is readily verified that $V^+$ is an $\hat{\mathcal{L}}$-submodule of $V$.
By the simplicity assumption on $V$, it follows that $V=V^+$.
Suppose for a moment that $\bar{P}^0$ is nontrivial. Then $P$ contains an $\hat{\mathcal{L}}$-submodule $W'$ such that
$W'/(W'\cap T)$ is a simple intermediate series $\hat{\mathcal{L}}$-module.
Let $v\in W'$ be any nonzero weight vector of nonzero weight. Then $v$ cannot lie in $V^+$, contradicting $V=V^+$. Therefore $\bar{P}^0=0$.

Next assume $\bar{P}^-$ is nontrivial. Define
\[
V^-=\Bigl\{v\in V\mid \exists\,N(v)\in\mathbb{Z}_+ \ \text{s.t.}\
L_{n,i}v=H_{n-\frac12,i}v=0,\ \forall\,n<N(v),\ i\in\mathbb{Z}\Bigr\}.
\]
A similar argument shows $V=V^-$.
We then obtain $V=V^+=V^-$, which implies that $V$ is trivial as an $\hat{\mathcal{L}}$-module.
This contradicts our initial hypothesis that $V$ is not uniformly bounded.
Hence, $\bar{P}^-$ is trivial.

We conclude that both $\bar{P}^0$ and $\bar{P}^-$ are trivial. Consequently, the support of $V$ is upper bounded,
so $V$ is a highest weight $\mathcal{L}$-module.

\begin{case}
The scalar action of $C_1$ is nonzero.
\end{case}

Since $C_1$ acts by a nonzero scalar, $V$ has no trivial subquotients as an $\hat{\mathcal{L}}$-module.  
We no longer need to factor out the maximal trivial submodule $T$, and several steps concerning intersections with $T$ can be omitted.
By a method similar to that used in Case $1$, we obtain $V=P=\bar{P}$ and $\bar{P}^0=0$.
Repeating the above line of argument yields the same conclusion that $V$ is a highest weight $\mathcal{L}$-module.
\end{proof}

Combining Theorem  \ref{thm:3.16} and 
 Theorem \ref{thm:4.5},  we obtain
\begin{theo}\label{thm:4.6}
Let $V$ be a simple Harish--Chandra $\mathcal{L}$-module. Then $V$ is either a highest weight module, a lowest weight module, or an evaluation module of the intermediate series $F_{a,b,c}(\lambda)$ or $\bar{F}_{a,b}(\lambda)$.
\end{theo}

As an application of Theorem~\ref{thm:4.5} to the truncated mirror Heisenberg-Virasoro algebra, we have the following corollary.

\begin{coro}\label{cor:4.6}
Let $n\ge 2$ be an integer, and let $V$ be a simple Harish-Chandra module over the truncated mirror Heisenberg-Virasoro algebra $\mathcal{L}(n)$. Then $V$ is either a highest  weight module, a lowest weight  module, or a uniformly bounded module.
\end{coro}
\begin{rema}
    From Corollaries \ref{cor:3.14} and  \ref{cor:4.6}, we also provide a complete classification of simple Harish-Chandra modules over the truncated mirror Heisenberg-Virasoro algebra $\mathcal{L}(n)$ for $n\geq2$. 
\end{rema}

\section*{Acknowledgements}
This work was supported by the National Natural Science Foundation of China (Grant No.~12361005). Chen would like to thank Prof.~Xiangqian Guo for providing a proof of Proposition \ref{lemm:2.2}. Part of this work was conducted during the authors' visit to the Chern Institute of Mathematics, Tianjin, China, from July 19 to 28, 2026. The authors are grateful to the institute and Prof.~Chengming Bai for their warm hospitality and support.

\section*{Authors' contributions}
All authors contributed equally to this work.

\section*{Data Availability Statement}
This manuscript has no associated data.
\section*{Conflicts of Interest} The authors declare that they have no conflicts of interest regarding
this work.

\bigskip

Haibo Chen

\vspace{2pt}
School of Science, Jimei University, Xiamen, Fujian 361021, China

\vspace{2pt}
hypo1025@jmu.edu.cn

\bigskip

Xiansheng Dai

\vspace{2pt}
School of Mathematical Sciences, Guizhou Normal University,
Guiyang 550001,  China

\vspace{2pt}
daisheng158@126.com

\bigskip
Yucai Su

\vspace{2pt}
School of Science, Jimei University, Xiamen, Fujian 361021, China

\vspace{2pt}
yucaisu@jmu.edu.cn


\small \begin{thebibliography}{9999}\vskip0pt
\parindent=2ex\parskip=-1pt\baselineskip=-1pt





 



\bibitem{B1}  Y. Billig, Jet modules, {\it Canad. J. Math.}, {\bf 59} (2007), 712-729.

\bibitem{BK}  Y. Billig,  K. Iohara, Classification of  uniformly bounded  simple modules
over a lattice Lie algebra of Witt type, {\it Canad. J. Math.}, {\bf 73} (2021), 417-440.

\bibitem{BF1} Y. Billig, V. Futorny, Classification of simple $W_n$-modules with finite-dimensional
weight spaces,   {\it J. Reine Angew. Math.}, {\bf  720} (2016), 199-216.

\bibitem{BF2} Y. Billig, V. Futorny, Classification of uniformly bounded  simple modules for solenoidal Lie
algebras,   {\it Israel J. Math.}, {\bf 222} (2017),  109-123.

\bibitem{BFIK} Y. Billig, V. Futorny, K. Iohara, I. Kashuba,
Classification of simple strong Harish-Chandra $W(m,n)$-modules, arXiv:2006.05618.


\bibitem{BGLZ} P.  Batra,  X.  Guo,  R.  L\"{u}, K.  Zhao,  Highest weight modules over the pre-exp-polynomial algebras, {\it  J. Algebra}, {\bf 322}  (2009), 4163-4180.
 

\bibitem{CLW} Y. Cai, R. L\"{u}, Y. Wang, Classification of simple Harish-Chandra modules for
map (super)algebras related to the Virasoro algebra, {\it J. Algebra}, {\bf 570} (2021), 397-415.

\bibitem{C}  H. Chen, Harish-Chandra modules over the higher rank $W$-algebra $W(2,2)$, {\it J. Algebra}, {\bf 694} (2026), 427-447. 

\bibitem{DCL}  M. Dilxat, L. Chen, D. Liu, Classification of simple Harish-Chandra modules over the Ovsienko-Roger
superalgebra, {\it Proc. Roy. Soc. Edinburgh Sect. A}, {\bf  154} (2024),  483-493.
 
 

 

\bibitem{GLZ1} X. Guo, R. L\"{u}, K. Zhao, Simple Harish-Chandra modules, intermediate series
modules, and Verma modules over the loop-Virasoro algebra, {\it  Forum Math.},
{\bf 23} (2011), 1029-1052.

 

\bibitem{K}
  V.G. Kac, Superconformal algebras and transitive group actions on quadrics, {\it Commun. Math.
Phys.}, {\bf 186}, (1997) 233-252.

 

\bibitem{LPXZ} 
D. Liu, Y. Pei, L. Xia,   K. Zhao,
Irreducible modules over the mirror Heisenberg-Virasoro algebra, {\it Commun.  Contemp. Math.},  {\bf 24} (2022),  2150026.


 
\bibitem{LG} G. Liu, X. Guo, Harish-Chandra modules over generalized Heisenberg-Virasoro algebras, {\it Israel J. Math.}, {\bf 204}  (2014),   447-468.

 



 



 

\bibitem{M2} O. Mathieu, Classification of Harish-Chandra modules over the Virasoro Lie algebra,
{\it Invent. Math.}, {\bf 107} (1992), 225-234.

 
\bibitem{MP}
C. Martin, A. Piard, Classification of the indecomposable bounded admissible modules
over the Virasoro Lie algebra with weight spaces of dimension not exceeding two, {\it Commun.
Math. Phys.}, {\bf 150} (1992), 465-493.

 
\bibitem{SA} A. Savage, 
Classification of irreducible quasifinite modules over map Virasoro algebras, {\it  Transformation Groups}, {\bf  17} (2012)  547–570. 

\bibitem{S0} Y. Su,   A classification of indecomposable $\mathfrak{sl}_2$($\C$)-modules and a conjecture of Kac on irreducible  modules over the Virasoro algebra, {\it J. Algebra}, {\bf 161} (1993), 33-46.

 


 

\bibitem{XL}
Y. Xue, R. L\"{u}, Simple weight modules with finite-dimensional weight spaces over Witt superalgebras, {\it J. Algebra}, {\bf 74}
(2021), 92-116.


\bibitem{WLP} Q. Wu, D. Liu, Y. Pei, Classification of the simple Harish-Chandra modules over the loop Neveu-Schwarz algebra,
{\it Isr. J. Math.}, (2025). https://doi.org/10.1007/s11856-025-2829-8.

\end{thebibliography}
\end{document}